\documentclass{amsart}

\usepackage{amscd}
\usepackage{amssymb}

\usepackage{color}

\usepackage{amsthm}

\newtheorem{thm}[equation]{Theorem}
\newtheorem{lem}[equation]{Lemma}
\newtheorem{lemma}[equation]{Lemma}
\newtheorem{cor}[equation]{Corollary}
\newtheorem{prop}[equation]{Proposition}

\newtheorem*{thm*}{Theorem}
\newtheorem*{prop*}{Proposition}
\newtheorem*{cor*}{Corollary}
\newtheorem*{lem*}{Lemma}
\newtheorem*{MT*}{Main Theorem}

\newtheorem*{ques*}{Question}

\theoremstyle{definition} %

\newtheorem*{defn*}{Definition}

\theoremstyle{remark} %

\newtheorem*{rmk*}{Remark}
\newtheorem*{rmks*}{Remarks}

\DeclareMathOperator{\Spin}{Spin}           % the Spin group

\DeclareMathOperator{\SL}{SL}
\DeclareMathOperator{\GL}{GL}

\DeclareMathOperator{\HSpin}{HSpin}

\newcommand{\SO}{\mathrm{SO}}

\newcommand{\Ga}{\mathbb{G}_\mathrm{a}}
\newcommand{\Gm}{\mathbb{G}_\mathrm{m}}

\DeclareMathOperator{\Lie}{Lie}
\DeclareMathOperator{\Ad}{Ad}
\newcommand{\ot}{\otimes}

\newcommand{\Z}{\mathbb{Z}}

\newcommand{\e}{\varepsilon}
\newcommand{\s}{\sigma}
\newcommand{\la}{\lambda}

\newcommand{\injects}{\hookrightarrow}

\DeclareMathOperator{\car}{char}
\DeclareMathOperator{\ed}{ed}
\DeclareMathOperator{\res}{res}
\DeclareMathOperator{\Spec}{Spec}
\DeclareMathOperator{\rank}{rank}

\renewcommand{\sl}{{\mathfrak{sl}}}

\renewcommand{\sl}{{\mathfrak{sl}}}
\newcommand{\so}{{\mathfrak{so}}}

\newcommand{\spin}{\mathfrak{spin}}
\newcommand{\hspin}{\mathfrak{hspin}}

\newcommand{\gbar}{\overline{g}}

\newcommand{\lie}{\mathfrak{g}}
\newcommand{\lsub}{\mathfrak{h}}

\newcommand{\trans}{\top}

\usepackage{hyperref}
\hypersetup{
  colorlinks=true, %set true if you want colored links
   linkcolor=red,  %choose some color if you want links to stand out
}
\usepackage[hyphenbreaks]{breakurl}

\newcommand{\g}{\mathfrak{g}}

\renewcommand{\t}{\mathfrak{t}}
\newcommand{\cc}{\mathfrak{c}}
\newcommand{\Gt}{\widetilde{G}}
\newcommand{\gt}{\tilde{\g}}
\newcommand{\Gtb}{\widetilde{\mathbf{G}}}

\numberwithin{equation}{section}

\begin{document}

\title{Spinors and essential dimension}
\author{Skip Garibaldi} 
\author{Robert M. Guralnick}
\thanks{Guralnick was partially supported by NSF grants DMS-1265297 and DMS-1302886.}

\dedicatory{With an appendix by Alexander Premet}

%\email{skip \text{at} member.ams.org}
%\url{http://www.mathcs.emory.edu/{\textasciitilde}skip}

\begin{abstract}
We prove that spin groups act generically freely on various spinor modules, in the sense of group schemes and in a way that does not depend on the characteristic of the base field.  As a consequence, we extend the surprising calculation of the essential dimension of spin groups and half-spin groups in characteristic zero by Brosnan--Reichstein--Vistoli (Annals of Math., 2010) and Chernousov--Merkurjev (Algebra \& Number Theory, 2014) to fields of characteristic different from 2.  We also complete the determination of generic stabilizers in spin and half-spin groups of low rank.
\end{abstract}

%\date{\tt Version of \today}


\maketitle

\section{Introduction}

The essential dimension of an algebraic group $G$ is, roughly speaking, the number of parameters needed to specify a $G$-torsor.
Since the notion was introduced in \cite{BuhRei} and \cite{RY}, there have been many papers calculating the essential dimension of various groups, such as \cite{MR1992016}, \cite{ChSe}, \cite{MR2358058},  \cite{MR2393078},   \cite{GiRei},   \cite{M:PGLp2}, \cite{MR2979508}, \cite{MR3039772},   etc.  (See \cite{M:bbk}, \cite{M:ed}, or \cite{Rei:ICM} for a survey of the current state of the art.)  For connected groups, the essential dimension of $G$ tends to be less than the dimension of $G$ as a variety; for semisimple adjoint groups this is well known\footnote{See \cite{GG:edp} for a proof that works regardless of the characteristic of the field.}.  Therefore, the discovery by Brosnan--Reichstein--Vistoli in \cite{BRV} that the essential dimension of the spinor group $\Spin_n$ grows exponentially as a function of $n$ (whereas $\dim \Spin_n$ is quadratic in $n$), was startling.  Their results, together with refinements for $n$ divisible by 4 in \cite{M:chile} and \cite{ChM:edspin}, determined the essential dimension of $\Spin_n$ for $n > 14$ if $\car k = 0$.  One goal of the present paper is to extend this result to all characteristics except $2$.

\subsection*{Generically free actions}
The source of the characteristic zero hypothesis in \cite{BRV} is that the upper bound relies on a fact about the action of spin groups on spinors that is only available in the literature in case the field $k$ has characteristic zero.  Recall that a group $G$ acting on a vector space $V$ is said to act \emph{generically freely} if there is a dense open subset $U$ of $V$ such that, for every $K \supseteq k$ and every $u \in U(K)$, the stabilizer in $G$ of $u$ is the trivial group scheme.  We prove:

\begin{thm} \label{MT}
Suppose  $n > 14$.  Then $\Spin_n$ acts generically freely on 
the spin representation if $n \equiv 1, 3 \bmod 4$; a half-spin representation if $n \equiv 2 \bmod 4$; or a direct sum of the vector representation and a half-spin representation if $n \equiv 0 \bmod 4$.
Furthermore, if $n \equiv 0 \bmod 4$ and $n \ge 20$, then $\HSpin_n$ acts generically freely on a half-spin representation.
\end{thm}

(We also compute the stabilizer of a generic vector for the values of $n$ not covered by Theorem \ref{MT}.  See below for precise statements.)

Throughout, we write $\Spin_n$ for the split spinor group, which is the simply connected cover (in the sense of linear algebraic groups) of the split group $\SO_n$.  To be precise, the \emph{vector representation} is the map $\Spin_n \to \SO_n$, which is uniquely defined up to equivalence unless $n = 8$.  For $n$ not divisible by 4, the kernel $\mu_2$ of this representation is the unique central $\mu_2$ subgroup of $\Spin_n$.  

For $n$ divisible by 4, the natural action of $\Spin_n$ on the spinors is a direct sum of two inequivalent representations, call them $V_1$ and $V_2$, each of which is called a \emph{half-spin} representation.  The center of $\Spin_n$ in this case contains two additional copies of $\mu_2$, namely the kernels of the half-spin representations $\Spin_n \to \GL(V_i)$, and we write $\HSpin_n$ for the image of $\Spin_n$ (the isomorphism class of which does not depend on $i$).  For $n \ge 12$, $\HSpin_n$ is not isomorphic to $\SO_n$.

Theorem \ref{MT} is known under the additional hypothesis that $\car k = 0$, see \cite[Th.~1]{AndreevPopov} for $n \ge 29$ and \cite{APopov} for $n \ge 15$.  The proof below is independent of the characteristic zero results, and so gives an alternative proof.

We note that Guerreiro proved that the generic stabilizer in the Lie algebra $\spin_n$, acting on a (half) spin representation, is central for $n = 22$ and $n \ge 24$, see Tables 6 and 9 of \cite{Guerreiro}.  At the level of group schemes, this gives the weaker result that the generic stabilizer is finite \'etale.  Regardless, we recover these cases quickly, see \S\ref{large.sec}; the longest part of our proof concerns the cases $n = 18$ and 20.

\subsection*{Generic stabilizer in $\Spin_n$ for small $n$}
For completeness, we list the stabilizer in $\Spin_n$ of a generic vector for $6 \le n \le 14$ in Table \ref{stabilizers}.   The entries for $n \le 12$ and $\car k \ne 2$ are from \cite{Igusa}; see sections \ref{small.2}--\ref{D7.2.sec} below for the remaining cases.  The case $n = 14$ is particularly important due to its relationship with the structure of 14-dimensional quadratic forms with trivial discriminant and Clifford invariant (see \cite{Rost:14.1}, \cite{Rost:14.2}, \cite{G:lens}, and \cite{M:invrr}), so we calculate the stabilizer in detail in that case.
\newcommand{\same}{\text{same}}
\begin{table}[tph]
\begin{center}
\[
\begin{array}{c|cc||c|cc} 
n&\car k \ne 2&\car k = 2&n&\car k \ne 2&\car k = 2 \\ \hline
6&(\SL_3)\cdot(\Ga)^3&\same& 
11&\SL_5&\SL_5 \rtimes \Z/2
\\
7&G_2&\same&
12&\SL_6&\SL_6 \rtimes \Z/2
\\
8&\Spin_7&\same&
13&\SL_3 \times \SL_3 &(\SL_3 \times \SL_3) \rtimes \Z/2
\\
9&\Spin_7&\same&
14&G_2 \times G_2&(G_2 \times G_2)\rtimes \Z/2
\\
10&(\Spin_7)\cdot (\Ga)^8&\same
\\
\end{array}
\]
\caption{Stabilizer sub-group-scheme in $\Spin_n$ of a generic vector in an irreducible (half) spin representation for small $n$.} \label{stabilizers}
\end{center}
\end{table}

For completeness, we also record the following.
\begin{thm} \label{D8}
Let $k$ be an algebraically closed field.  The stabilizer in $\HSpin_{16}$ of a generic vector in a half-spin representation is isomorphic to $(\Z/2)^4 \times (\mu_2)^4$.
\end{thm}

The proof when $\car k \ne 2$ is short, see Lemma \ref{Spin16}.  The case of $\car k = 2$ is treated in an appendix by Alexander Premet.  (Eric Rains has independently proved this result.)

\subsection*{Essential dimension}
We recall the definition of essential dimension.  For an extension $K$ of a field $k$ and an element $x$ in the Galois cohomology set $H^1(K, G)$, we define $\ed(x)$ to be the minimum of the transcendence degree of $K_0/k$ for $k \subseteq K_0 \subseteq K$ such that $x$ is in the image of $H^1(K_0, G) \to H^1(K, G)$.  The \emph{essential dimension} of $G$, denoted $\ed(G)$, is defined to be $\max \ed(x)$ as $x$ varies over all extensions $K/k$ and all $x \in H^1(K, G)$.  There is also a notion of \emph{essential $p$-dimension} for a prime $p$.  The essential $p$-dimension $\ed_p(x)$ is the minimum of $\ed(\res_{K'/K} x)$ as $K'$ varies over finite extensions of $K$ such that $p$ does not divide $[K':K]$, where $\res_{K'/K} \!: H^1(K, G) \to H^1(K', G)$ is the natural map.  The essential $p$-dimension of $G$, $\ed_p(G)$, is defined to be the minium of $\ed_p(x)$ as $K$ and $x$ vary; trivially, $\ed_p(G) \le \ed(G)$ for all $p$ and $G$, and $\ed_p(G) = 0$ if for every $K$ every element of $H^1(K, G)$ is killed by some finite extension of $K$ of degree not divisible by $p$.  

Our Theorem \ref{MT} gives upper bounds on the essential dimension of $\Spin_n$ and $\HSpin_n$ regardless of the characteristic of $k$.  Combining these with the results of \cite{BRV}, \cite{M:chile}, \cite{ChM:edspin}, and \cite{Loetscher:fiber} quickly gives the following, see \S\ref{ed.sec} for details.

\begin{cor} \label{ed}
For $n > 14$ and $\car k \ne 2$,  
\[
\ed_2(\Spin_n) = \ed(\Spin_n) = \begin{cases}
2^{(n-1)/2} - \frac{n(n-1)}2 & \text{if $n \equiv 1,3 \bmod{4}$;} \\
2^{(n-2)/2} - \frac{n(n-1)}2 & \text{if $n \equiv 2 \bmod{4}$; and} \\
2^{(n-2)/2} - \frac{n(n-1)}2 + 2^m & \text{if $n \equiv 0 \bmod{4}$}
\end{cases}
\]
where $2^m$ is the largest power of $2$ dividing $n$ in the final case.  For $n \ge 20$ and divisible by $4$, 
\[
\ed_2(\HSpin_n) = \ed(\HSpin_n) = 2^{(n-2)/2} - \frac{n(n-1)}2.
\] 
\end{cor}

Although Corollary \ref{ed} is stated and proved for split groups, it quickly implies analogous results for non-split forms of these groups, see \cite[\S4]{Loetscher:fiber} for details.

Combining the corollary with the calculation of $\ed(\Spin_n)$ for $n \le 14$ by Markus Rost in \cite{Rost:14.1} and \cite{Rost:14.2} (see also \cite{G:lens}), we find for $\car k \ne 2$:
\[
\begin{array}{c|rrrrrrrrrrrrrrrr} 
n&5&6&7&8&9&10&11&12&13&14&15&16&17&18&19&20 \\ \hline
\ed(\Spin_n)&0&0&4&5&5&4&5&6&6&7&23&24&120&103&341&326
\end{array}
\]

\subsection*{Notation}
Let $G$ be an affine group scheme of finite type over a field $k$, which we assume is algebraically closed and of characteristic different from 2.  (If $G$ is additionally smooth, then we say that $G$ is an \emph{algebraic group}.)  If $G$ acts on a variety $X$, the stabilizer $G_x$ of an element $x \in X(k)$ is a sub-group-scheme of $G$ with $R$-points
\[
G_x(R) = \{ g \in G(R) \mid gx = x \}
\]
for every $k$-algebra $R$.

If $\Lie(G) = 0$ then $G$ is finite and \'etale.  If additionally $G(k) = 1$, then $G$ is the trivial group scheme $\Spec k$.

For a representation $\rho \!: G \to \GL(V)$ and elements $g \in G(k)$ and $x \in \Lie(G)$, we denote the fixed spaces by $V^g := \ker(\rho(g) - 1)$ and $V^x := \ker( \mathrm{d}\rho(x))$.

We use fraktur letters such as $\g$, $\spin_n$, etc., for the Lie algebras $\Lie(G)$, $\Lie(\Spin_n)$, etc.

{\small\subsection*{Acknowledgements} We thank Alexander Merkurjev and Zinovy Reichstein for helpful comments, and for posing the questions answered in Proposition \ref{D7.2} and Theorem \ref{MT}.}

%%%%%%%%%%%%%%%%%%%%%%%%%%%%%%%%%%%%%%%%%%%%%%%%%%%%%%%%%%%%%%%%%%%%%%%%
\section{Fixed spaces of elements}

Fix some $n \ge 6$.
Let $V$ be a (half) spin representation for $\Spin_n$, of dimension $2^{\lfloor (n-1)/2 \rfloor}$.

\begin{prop} \label{espaces}
For $n \ge 6$: 
\begin{enumerate}
\item \label{espaces.lie} For all noncentral $x \in \spin_n$, $\dim V^x \le \frac34 \dim V$.  
\item \label{espaces.grp} For all noncentral $g \in \Spin_n$, $\dim V^g \le \frac34 \dim V$.
\end{enumerate}
If $n > 8$,  $\car k \ne 2$, and $g \in \Spin_n$ is noncentral semisimple, then $\dim V^g \le \frac58 \dim V$.
\end{prop}

In the proof, in case $\car k \ne 2$, we view $\SO_n$ as the group of matrices
\[
\SO_n(k) = \{ A \in \SL_n(k) \mid SA^{\trans}S = A^{-1} \},
\]
where $S$ is the matrix 1's on the ``second diagonal'', i.e., $S_{i,n+1-i} = 1$ and the other entries of $S$ are zero.  The intersection of the diagonal matrices with $\SO_n$ are a maximal torus.  For $n$ even, one finds elements of the form $(t_1, t_2, \ldots, t_{n/2}, t_{n/2}^{-1}, \ldots, t_1^{-1})$, and we abbreviate these as $(t_1, t_2, \ldots, t_{n/2}, \ldots)$.  Explicit formulas  for a triality automorphism $\s$ of $\Spin_8$ of order 3 are given in \cite[\S1]{G:iso}, and for $g = (t_1, t_2, t_3, t_4, \ldots) \in \Spin_8$ the elements $\s(g)$ and $\s^2(g)$ have images in $\SO_8$
\begin{gather} \label{triality}
\e \left( \sqrt{\tfrac{t_1 t_2 t_3}{t_4}}, \sqrt{\tfrac{t_1 t_4 t_2}{t_3}}, 
\sqrt{ \tfrac{t_1 t_3 t_4}{t_2}}, \sqrt{\tfrac{t_1}{t_2t_3t_4}} , \ldots \right) \quad \text{and} \\
  \e \left( \sqrt{t_1 t_2 t_3 t_4}, \sqrt{\tfrac{t_1 t_2}{t_3 t_4}}, 
\sqrt{\tfrac{t_1t_3}{t_2t_4}}, \sqrt{\tfrac{t_2 t_3}{t_1 t_4}}, \ldots\right), \notag
\end{gather}
where $\e = \pm 1$ is the only impecision in the expression.

\begin{proof}
For \eqref{espaces.lie}, in the Jordan decomposition $x = s + n$ where $s$ is semisimple, $n$ is nilpotent, and $[s,n] = 0$, we have $V^x \subseteq V^s \cap V^n$, so it suffices to prove \eqref{espaces.lie} for $x$ nilpotent or semisimple.

Suppose first that $x$ is a root element.  If $n = 6$, then $\spin_n \cong \sl_4$ and $V$ is the natural representation of $\sl_4$, so  we have the desired equality.
For $n > 6$,  the module restricted to $\so_{n-1}$ is
either irreducible or the direct sum of two half spins and so the result
follows.

If $x$ is nonzero nilpotent, then we may replace $x$ by a root element in the closure of $(\Ad G)x$.  If $x$ is noncentral semisimple, choose a root subgroup $U_\alpha$ of $\SO_n$ belonging to a Borel subgroup $B$ such that $x$ lies in $\Lie(B)$ and does not commute with $U_\alpha$.  Then for all $y \in \Lie(U_\alpha)$ and all scalars $\la$, $x+\la y$ is in the same $\Ad(\SO_n)$-orbit as $x$ and $y$ is in the closure of the set of such elements; replace $x$ with $y$.  If $x$ is nonzero nilpotent, then root elements are in the closure of $\Ad(G)x$. 

(In case $n$ is divisible by 4, the natural map $\spin_n \to \hspin_n$ is an isomorphism on root elements.  It follows that for noncentral $x \in \hspin_n$, $\dim V^x \le \frac34 \dim V$ by the same argument.)

\smallskip

For \eqref{espaces.grp}, 
we may assume that $g$ is unipotent or semisimple. 
If $g$ is unipotent, then by taking closures, we may pass to root elements
and argue as for $x$ in the Lie algebra.

If $g$ is semisimple,  we actually prove a slightly stronger result:
\emph{all} eigenspaces have dimension at most $\frac34 \dim V$.

Suppose now that $n$ is even.  The image of $g$ in $\SO_n$ can be viewed as an element of $\SO_{n-2} \times \SO_2$, where it has eigenvalues $(a,a^{-1})$ in $\SO_2$.  Replacing if necessary $g$ with a multiple by an element of the center of $\Spin_n$, we may assume that $g$ is in the image of $\Spin_{n-2} \times \Spin_2$. 
Then  $V = V_1 \oplus V_2$ where the $V_i$ are distinct half spin modules for
$\Spin_{n-2}$ and the $\Spin_2$ acts on each (since they are distinct
and $\Spin_2$ commutes with $\Spin_{n-2}$).   By induction every eigenspace
of $g$  has dim at most $\frac34 \dim V_i$ and the $\Spin_2$ component of $g$ acts as
a scalar, so this is preserved.

If $n$ is odd, then the image of $g$ in $\SO_n$ has eigenvalue 1 on the natural module, so is contained in a $\SO_{n-1}$ subgroup.  Replacing if necessary $g$ with $gz$ for some $z$ in the center of $G$, we may assume that $g$ is in the image of $\Spin_{n-1}$ and the claim follows by induction.

\smallskip

For the final claim, view $g$ as an element in the image of $(g_1, g_2) \in \Spin_8 \times \Spin_{n-8}$.  
The result is clear unless in some 8-dimensional image of $\Spin_8$, $g_1$ has diagonal image $(a, 1, 1, 1, \ldots) \in \SO_8$.  We may assume that the image of $g$ has prime order in $\GL(V)$.

If $g_1$ has odd order or order 4, then as in \eqref{triality}, on the other two 8 dimensional representations of $\Spin_8$ it has no fixed space and each eigenspace of dimension at most 4.  So on the sum of any two of these representations the largest eigenspace is at most 10-dimensional (out of 16), and the claim follows.

If $g_1$ has order 2, then it maps to an element of order 4 in the other two 8-dimensional representations via \eqref{triality} and again the same argument applies.
\end{proof}

The proposition will feed into the following elementary lemma, which resembles \cite[Lemma 4]{AndreevPopov} and \cite[\S3.3]{Guerreiro}.
\begin{lem} \label{mother.lem}
Let $V$ be a representation of a semisimple algebraic group $G$ over an algebraically closed field $k$.
\begin{enumerate}
\item \label{mother.grp} If for every unipotent $g \in G$ and every noncentral semisimple $g \in G$ of prime order we have
\begin{equation} \label{mother.1}
\dim V^g + \dim g^G < \dim V,
\end{equation}
then for generic $v \in V$, $G_v(k)$ is central in $G(k)$.

\item \label{mother.lie} Suppose $\car k = p > 0$ and let $\lsub$ be a $G$-invariant subspace of $\g$.  If for every nonzero $x \in \g \setminus \lsub$ such that $x^{[p]} \in \{ 0, x \}$ we have
\begin{equation} \label{mother}
\dim V^x + \dim (\Ad(G)x) < \dim V,
\end{equation}
then for generic $v \in V$, $\Lie(G_v) \subseteq \lsub$.
\end{enumerate}
\end{lem}

We will apply this to conclude that $G_v$ is the trivial group scheme for generic $v$, so the hypothesis on $\car k$ in \eqref{mother.lie} is harmless.  When $\car k = 0$, the conclusion of \eqref{mother.grp} suffices.  

\begin{proof}
For \eqref{mother.grp}, see \cite[\S10]{GG:simple} or adjust slightly the following proof of \eqref{mother.lie}.  For $x \in \lie$, define
\[
V(x) := \{ v \in V \mid \text{there is $g \in G(k)$ s.t.~$xgv = 0$} \} = \bigcup_{g \in G(k)} gV^x.
\]
Define $\alpha \!: G \times V^x \to V$ by $\alpha(g,w) = gw$, so the image of $\alpha$ is precisely $V(x)$.  The fiber over $gw$ contains $(gc^{-1}, cw)$ for $\Ad(c)$ fixing $x$, and so $\dim V(x) \le \dim (\Ad(G)x) + \dim V^x$.

Let $X \subset \lie$ be the set of nonzero  $x \in \lie \setminus \lsub$ such that $x^{[p]} \in \{ 0, x \}$; it is a union of finitely many $G$-orbits.  (Every toral element --- i.e., $x$ with $x^{[p]} = x$ --- belongs to $\Lie(T)$ for a maximal torus $T$ in $G$ by \cite{BoSp1}, and it is obvious that there are only finitely many conjugacy classes of toral elements in $\Lie(T)$.)  Now $V(x)$ depends only on the $G$-orbit of $X$ (because $V^{\Ad(g)x} = gV^x$), so the union $\cup_{x \in X} V(x)$ is a finite union.  As $\dim V(x) < \dim V$ by the previous paragraph, the union $\cup V(x)$ is contained in a proper closed subvariety $Z$ of $V$, and for every $v$ in the (nonempty, open) complement of $Z$, $\lie_v$ does not meet $X$.

For each $v \in (V \setminus Z)(k)$ and each $y \in \lie_v$, we can write $y$ as
\[
y = y_n + \sum_{i = 1}^r \alpha_i y_i, \quad \text{$[y_n, y_i] = [y_i, y_j] = 0$ for all $i, j$}
\]
such that $y_1, \ldots, y_r \in \lie_v$ are toral, and $y_n \in \lie_v$ satisfies $y_n^{[p]} = 0$, see \cite[Th.~3.6(2)]{StradeF}.  Thus  $y_n$ and the $y_1, \ldots, y_r$ are in $\lsub$ by the previous paragraph.
\end{proof}

Note that, in proving Theorem \ref{MT}, we may assume that $k$ is algebraically closed (and so this hypothesis in Lemma \ref{mother.lem} is harmless).  Indeed, suppose $G$ is an algebraic group acting on a vector space $V$ over a field $k$.  Fix a basis $v_1, \ldots, v_n$ of $V$ and consider the element $\eta := \sum t_i v_i \in V \ot k(t_1, \ldots, t_n) = V \ot k(V)$ for indeterminates $t_1, \ldots, t_n$; it is a sort of generic point of $V$.  Certainly, $G$ acts generically freely on $V$ over $k$ if and only if the stabilizer $(G \times k(V))_v$ is the trivial group scheme, and this statement is unchanged by replacing $k$ with an algebraic closure.  That is, $G$ acts generically freely on $V$ over $k$ if and only if $G \times K$ acts generically freely on $V \ot K$ for $K$ an algebraic closure of $k$.

%%%%%%%%%%%%%%%%%%%%%%%%%%%%%%%%%%%%%%%%%%%%%%%%%%%%%%%%%%%%%%%%%%%%%%%%
\section{Proof of Theorem \ref{MT} for $n > 20$} \label{large.sec}

Suppose $n > 2$, and 
put $V$ for a (half) spin representation of $\Spin_n$.  Recall that 
\[
\dim \Spin_n = r(2r-1) \quad \text{and} \quad \dim V = 2^{r-1} \quad \text{if $n = 2r$}
\]
whereas
\[
\dim \Spin_n = 2r^2 +r \quad \text{and} \quad \dim V = 2^r \quad \text{if $n = 2r+1$}
\]
and in both cases $\rank \Spin_n = r$.  Proposition \ref{espaces} gives an upper bound on $\dim V^g$ for noncentral $g$, and certainly the conjugacy class of $g$ has dimension at most $\dim \Spin_n - r$.  If we assume $n \ge 21$ and apply these, we obtain \eqref{mother.1}
and consequently the stabilizer $S$ of a generic $v \in V$ has $S(k)$ central in $\Spin_n(k)$.  Repeating this with the Lie algebra $\spin_n$ (and $\lsub$ the center of $\spin_n$) we find that $\Lie(S)$ is central in $\spin_n$.  For $n$ not divisible by 4, the representation $\Spin_n \to \GL(V)$ restricts to a closed embedding on the center of $\Spin_n$, so $S$ is the trivial group scheme as claimed in Theorem \ref{MT}.  

For $n$ divisible by 4, we conclude that $\HSpin_n$ acts generically freely on $V$ (using that Proposition \ref{espaces}\ref{espaces.lie} holds also for $\hspin_n$).  
 As the kernel $\mu_2$ of $\Spin_n \to \HSpin_n$ acts faithfully on the vector representation $W$, it follows that $\Spin_n$ acts generically freely on $V \oplus W$, completing the proof of Theorem \ref{MT} for $n > 20$.

%%%%%%%%%%%%%%%%%%%%%%%%%%%%%%%%%%%%%%%%%%%%%%%%%%%%%%%%%%%%%%%%%%%%%%%%
\section{Proof of Theorem \ref{MT} for  $n \le 20$ and characteristic $\ne 2$}

In this section we assume that $\car k \ne 2$, and in particular the Lie algebra $\spin_n$ (and $\hspin_n$ in case $n$ is divisible by 4) is naturally identified with $\so_n$. 

\subsection*{Case $n = 18$ or $20$}
Take $V$ to be a half-spin representation of $G = \Spin_n$ (if $n = 18$) or $G = \HSpin_n$ (if $n = 20$).  To prove Theorem \ref{MT} for these $n$, it suffices to prove that $G$ acts generically freely on $V$, which we do by verifying the inequalities \eqref{mother.1} and \eqref{mother}.

\subsubsection*{Nilpotents and unipotents} Let $x \in \g$ with $x^{[p]}=  0$.  The argument for unipotent elements of $G$ is essentially identical (as we assume $\car k \ne 2$) and we omit it.

If, for a particular $x$, we find that the centralizer of $x$ has dimension $> 89$ (if $n = 18$) or $> 62$ (if $n = 20$), then $\dim(\Ad(G)x) < \frac14 \dim V$ and we are done by Proposition \ref{espaces}.

For $x$ nilpotent, 
the most interesting case is where $x$ is has partition $(2^{2t}, 1^{n-2t})$ for some $t$.  If $n = 20$, then such a class has centralizer of dimension at least 100, and we are done.  If $n = 18$, we may assume by similar reasoning that $t = 3$ or 4.  The centralizer of $x$ has dimension $\ge 81$, so $\dim(\Ad(G)x) \le 72$.  We claim that $\dim V^x \le 140$; it suffices to prove this for an element with $t = 3$, as the element with $t = 4$ specializes to it.  View it as an element in the image of $\so_9 \times \so_9 \to \so_{18}$ where the first factor has partition $(2^4, 1)$ and the second has partition $(2^2, 1^5)$.  Now, triality on $\so_8$ sends elements with partition $2^4$ to elements with partition $2^4$ and $(3, 1^5)$ --- see for example \cite[p.~97]{CMcG} --- consequently the $(2^4, 1)$ in $\so_9$ acts on the spin representation of $\so_9$ as a $(3, 2^4, 1^5)$.  Similarly, the $(2^2, 1^5)$ acts on the spin representation of $\so_9$ as $(2^4, 1^8)$.  The action of $x$ on the half-spin representation of $\so_{18}$ is the tensor product of these, and we find that $\dim V^x \le 140$ as claimed.

Suppose $x$ is nilpotent and has a Jordan block of size at least 5.  An element with partition $(5,1)$ in $\so_6$ is a regular nilpotent in $\sl_4$ with 1-dimensional kernel.  Using the tensor product decomposition as in the proof of Proposition \ref{espaces}, we deduce that an element $y \in \so_n$ with partition $(5, 1^{n-5})$ has $\dim V^y \le \frac14 \dim V$, and consequently by specialization $\dim V^x \le \frac14 \dim V$.  As $\dim(\Ad(G)x) \le \dim G - \rank G < \frac34 \dim V$, the inequality is verified for this $x$.

Now suppose $x$ is nilpotent and all Jordan blocks have size at most 4, so it is a specialization of $(4^4, 1^2)$ if $n = 18$ or $4^5$ if $n = 20$.  These classes have centralizers of dimension 41 and 50 respectively, hence $\dim (\Ad(G) x) < \frac12 \dim V$.  If $x$ has at least two Jordan blocks of size at least 3, then $x$ specializes to $(3^2, 1^{n-6})$; as triality sends elements with partition $(3^2, 1^2)$ to elements with the same partition, we find $\dim V^x \le \frac12 \dim V$.  We are left with the case where $x$ has partition $(3, 2^{2t}, 1^{n-2t-3})$ for some $t$.  If $t = 0$, then the centralizer of $x$ has dimension 121 or 154 and we are done.  If $t > 0$, then $x$ specializes to $y$ with partition $(3, 2^2, 1^{n-7})$.  As triality on $\so_8$ leaves the partition $(3, 2^2, 1)$ unchanged, we find $\dim V^x \le \dim V^y \le \frac12 \dim V$, as desired, completing the verification of \eqref{mother} for $x$ nilpotent.

\subsubsection*{Semisimple elements in $\Lie(G)$}
For $x \in \so_n$ semisimple, the most interesting case is when $x$ is diagonal with entries $(a^t, (-a)^t, 0^{n-2t})$ where exponents denote multiplicity and $a \in k^\times$.  The centralizer of $x$ is $\GL_t \times \SO_{n-2t}$, so 
$\dim(\Ad(\SO_n)x) = \textstyle\binom{n}{2} - t^2 - \textstyle\binom{n-2t}{2}$.
This is less than $\frac14 \dim V$ for $n = 20$, settling that case.  For $n = 18$, if $t = 1$ or 2, $x$ is in the image of an element $(a,-a,0,0)$ or $(a/2, a/2, -a/2, -a/2)$ in $\sl_4 \cong \so_6$, and the tensor product decomposition gives that $\dim V^x \le \frac12 \dim V$ and again we are done.  If $t > 2$, we consider a nilpotent $y = \left( \begin{smallmatrix} 0 & Y \\ 0 & 0 \end{smallmatrix} \right)$ not commuting with $x$ where $Y$ is 9-by-9 and $y$ specializes to a nilpotent $y'$ with partition $(2^4, 1^8)$.  Such a $y'$ acts on $V$ as 16 copies of $(3, 2^4, 1^5)$, hence $\dim V^{y'} = 160$.  By specializing $x$ to $y$ as in the proof of Proposition \ref{espaces}, we find $\dim V^x \le 160$ and again we are done.

\subsubsection*{Semisimple elements in $G$}
Let $g \in G(k)$ be semisimple, non-central, and of prime order.  If $n =  20$, then $\dim g^G \le 180 < \frac38 \dim V$ and we are done by Proposition \ref{espaces}.  
So assume $n = 18$.   If we find that the centralizer of $g$ has dimension $> 57$, then $\dim g^G < \frac38 \dim V$ and we are done by Proposition \ref{espaces}.

If $g$ has order 2, then it maps to an element of order 2 in $\SO_{18}$ whose centralizer is no smaller than $\SO_8 \times \SO_{10}$ of dimension 73, and we are done.  So assume $g$ has odd prime order.  We divide into cases depending on the image $\gbar \in \SO_{18}$ of $g$.

If $\gbar$ has at least 5 distinct eigenvalues, then either it has at least 6 distinct eigenvalues $a, a^{-1}, b, b^{-1}, c, c^{-1}$, or it has 4 distinct eigenvalues that are not equal to 1, and the remaining eigenvalue is 1.  In the latter case set $c = 1$.  View $g$ as the image of $(g_1, g_2) \in \Spin_6 \times \Spin_{12}$ where $g_1$ maps to a diagonal $(a, b, c, c^{-1}, b^{-1}, a^{-1})$ in $\SO_6$, a regular semisimple element.  Therefore, the eigenspaces of the image of $g_1$ under the isomorphism $\Spin_6 \cong \SL_4$ are all 1-dimensional and the tensor decomposition argument shows that $\dim V^g \le \frac14 \dim V$.  As $\dim g^G \le 144 < \frac34 \dim V$, we are done in this case.

If $\gbar$ has exactly 4 eigenvalues, then the centralizer of $\gbar$ is at least as big as $\GL_4 \times \GL_5$ of dimension 41, so $\dim g^G \le 112 < \frac12 \dim V$.  Viewing $g$ as the image of $(g_1, g_2) \in \Spin_8 \times \Spin_{10}$ such that the image $\gbar_1$ of $g_1$ in $\SO_8$ exhibits all 4 eigenvalues, then $\gbar_1$ has eigenspaces all of dimension 2 or of dimensions 3, 3, 1, 1.  Considering the possible images of $\gbar_1$ as in \eqref{triality}, each eigenspaces in each of the 8-dimensional representations is at most 4, so $\dim V^g \le \frac12 \dim V$ and this case is settled.

In the remaining case, $\gbar$ has exactly 2 nontrivial (i.e., not 1) eigenvalues $a, a^{-1}$.  If 1 is not an eigenvalue of $\gbar$, then the centralizer of $\gbar$ is $\GL_9$ of dimension 81, and we are done.  If the eigenspaces for the nontrivial eigenvalues are at least 4-dimensional, then we can take $g$ to be the image of $(g_1, g_2) \in \Spin_{10} \times \Spin_8$ where $g_1$ maps to $(a,a,a,a,1, \ldots) \in \SO_{10}$.  The images of $(a,a,a,a,\ldots) \in \SO_8$  as in \eqref{triality} are $(a,a,a,a^{-1},\ldots)$ and $(a^2, 1, 1, 1, \ldots)$, so the largest eigenspace of $g_1$ on a half-spin representation is 6, so $\dim V^g \le \frac38 \dim V$.  As the conjugacy class of a regular element has dimension $144 < \frac58 \dim V$, this case is complete.
Finally, if $\gbar$ has eigenspaces of dimension at most 2 for $a$, $a^{-1}$, then $\dim g^G \le 58 < \frac38 \dim V$ and the $n =18$ case is complete.

%%%%%%%%%%%%%%%%%%%%%%%%%%%%%%%%%%%%%%%%%%%%%%%%%%%%%%%%%%%%%
\subsection*{Case $n = 17$ or $19$}  For $n = 17$ or $19$, the spin representation of $\Spin_n$ can be viewed as the restriction of a half-spin represenation of the overgroup $\HSpin_{n+1}$.  We have already proved that this representation of $\HSpin_{n+1}$ is generically free.

\subsection*{Case $n = 15$ or $16$} We use the general fact:
\begin{lem} \label{finite}
Let $G$ be a quasi-simple algebraic group and $H$ a proper closed subgroup of $G$ and $X$ finite.
Then for generic $g \in G$, $H \cap gXg^{-1} = H \cap X \cap Z(G)$. 
\end{lem}

\begin{proof}
For each $x \in X \setminus Z(G)$, note that $W(x):= \{g \in G \mid x^g \in H\}$
is a proper closed subvariety of $G$ and, since $X$ is finite, $\cup W(x)$ is also proper closed.
Thus for an open subset of $g$ in $G$,  $g(X \setminus Z(G))g^{-1}$ does not meet $H$.
\end{proof}

\begin{lem} \label{Spin16}
Let $G = \HSpin_{16}$ and $V$ a half-spin representation over an algebraically closed field $k$ of characteristic $\ne 2$.
The stabilizer of a generic vector in $V$ is isomorphic to $(\Z/2)^8$, as a group scheme.
\end{lem}

\begin{proof}
Consider $\Lie(E_8) = \Lie(G) \oplus V$ where these are the eigen\-spaces
of an involution in $E_8$.   That involution inverts a maximal torus $T$ of $E_8$ and so 
there is maximal Cartan subalgebra $\mathfrak{t} = \Lie(T)$ on which the involution acts as $-1$.  As $E_8$ is smooth and adjoint, for a generic element $\tau \in \mathfrak{t}$, the centralizer $C_{E_8}(\tau)$ has identity component $T$ by \cite[XIII.6.1(d), XIV.3.18]{SGA3.2}, and in fact equals $T$ by \cite[Prop.~2.3]{GG:edp}.  Since $\mathfrak{t}$ misses $\Lie(G)$,  the annihilator of $\tau$ in $\Lie(G)$ is $0$ as claimed. Furthermore, $G_v(k) = T(k) \cap G(k)$, i.e., the elements of $T(k)$ that commute with the involution, so $G_v(k) \cong \mu_2(k)^8$.
\end{proof}

%The proof used that $\car k \ne 2$.  For $\car k = 2$ the generic stabilizer in $\Spin_{16}$ has been determined by Eric Rains \cite{Rains:D8}.

\begin{cor}If $\car k \ne 2$, then
$\Spin_{15}$ acts generically freely on V.
\end{cor}

\begin{proof}
Of course the Lie algebra does because this is true for $\Lie(\Spin_{16})$.

For the group,  a generic stabilizer is $\Spin_{15} \cap X$ where $X$ is a generic stabilizer
in $\Spin_{16}$.  Now $X$ is finite and meets the center of $\Spin_{16}$ in the kernel of $\Spin_{16} \to \HSpin_{16}$, whereas $\Spin_{15}$ injects in to $\HSpin_{16}$.  Therefore, by Lemma \ref{finite} a generic conjugate of $X$ intersect $\Spin_{15}$ is trivial.
\end{proof}

\begin{cor} If $\car k \ne 2$, then $\Spin_{16}$ acts generically freely on $V \oplus W$, where $V$ is a half-spin
and $W$ is the natural ($16$-dimensional) module.
\end{cor}

\begin{proof}
Now the generic stabilizer is already $0$ for the Lie algebra on $V$ whence on $V \oplus W$.

In the group $\Spin_{16}$, a generic stabilizer is conjugate to $X^g \cap \Spin_{15}$ where $X$ is the finite
stabilizer on $V$ and as in the proof of the previous corollary, this is generically trivial.
\end{proof}

%%%%%%%%%%%%%%%%%%%%%%%%%%%%%%%%%%%%%%%%%%%%%%%%%%%%%%%%%%%%%%%%%%%%%%%%
\section{Proof of Theorem \ref{MT} for $n \le 20$ and characteristic 2} \label{spin20.2}

%To complete the proof of Theorem \ref{MT}, it remains to prove, in case $\car k = 2$, that the following representations $G \to \GL(V)$ are generically free:
%\begin{enumerate}
%\item $G = \Spin_{15}$, $\Spin_{17}$, $\Spin_{19}$ and $V$ is a spin representation.
%\item $G = \Spin_{18}$ and $V$ is a half-spin representation.
%\item $G = \Spin_{16}$ or $\Spin_{20}$ and $V$ is a direct sum of the vector representation and a half-spin representation.
%\item $G = \HSpin_{20}$ and $V$ is a half-spin representation.
%\end{enumerate}
%
%By applying the same techniques as in the previous section or by referring to \cite{GurLawther}, we see that the group of $k$-points $G_v(k)$ of the stabilizer of a generic $v \in V$ is the trivial group.  It remains to check that $\Lie(G_v) = 0$, which can be checked computationally as follows.  Pick any field $F$ of characteristic 2 and any $w \in V$ and compute the stabilizer $\g_w$.  (This can be done easily in various modern computer algebra systems.)  In each case, one can find a $w$ such that $\g_w = 0$, completing the proof of Theorem \ref{MT}.  $\hfill\qed$

To complete the proof of Theorem \ref{MT}, it remains to prove, in case $\car k = 2$, that the following representations $G \to \GL(V)$ are generically free:
\begin{enumerate}
\item $G = \Spin_{15}$, $\Spin_{17}$, $\Spin_{19}$ and $V$ is a spin representation.
\item $G = \Spin_{18}$ and $V$ is a half-spin representation.
\item $G = \Spin_{16}$ or $\Spin_{20}$ and $V$ is a direct sum of the vector representation and a half-spin representation.
\item $G = \HSpin_{20}$ and $V$ is a half-spin representation.
\end{enumerate}

Since we are in bad characteristic, the class of unipotent and nilpotent elements are more complicated.  On the other hand,
since we are in a fixed small characteristic and the dimensions of the modules and Lie algebras are relatively small, one can 
actually do some computations.

In particular, we check that in each case that there exists a $v \in V$ over the field of $2$ elements such that $\Lie(G_v)=0$.
(This can be done easily in various  computer algebra systems.)    It follows that the same is true over any field
of characteristic $2$.   Since the set of $w \in V$ where $\Lie(G_w)=0$ is an open subvariety of $V$, this shows
that $\Lie(G_w)$ is generically $0$. 

It remains to show that the group of $k$-points $G_v(k)$ of the stabilizer of a generic $v \in V$ is the trivial group.  

First consider $G=\Spin_{16}$.  By Lemma \ref{finite}, it suffices to show that, for generic $w$ in a half-spin representation $W$, $G_w(k)$ is finite, which is true by the appendix.  Alternatively, the finiteness of $G_w(k)$ was proved in \cite{GurLawther} by working in $\Lie(E_8) =  \hspin_{16} \oplus W$ and exhibiting a regular nilpotent of $\Lie(E_8)$
in $W$ whose stabilizer in $\hspin_{16}$ is trivial.   Since the set of $w$ where $(\Spin_{16})_w(k)$ is finite is open, the result follows.

%on $W$ a half-spin representation of dimension $128$.    We note that
%in this case,  $\Lie(G_w) \ne 0$ for any $w$ (and indeed is generically $4$-dimensional).   %See \cite{Guerreiro}.   
%We first claim:
%
%\begin{lem}   $(\Spin_{16})_w(k)$ is finite for generic $w \in W$.
%\end{lem}
%
%\begin{proof}  This is proved in \end{proof}
%
%A more precise result has been proved independently by Rains  and Premet (private communication).   

%Now by Lemma \ref{finite},  we see that $\Spin_{16}$ on the direct sum of the vector representation and a half-spin representation is
%generically free.  

Similarly, $\Spin_{15}$ acts generically freely on the spin representation.   

As in the previous section, it suffices to show that for $G$ one of $\HSpin_{20}$ and $\Spin_{18}$ and $V$ a half spin representation, 
$G_v(k)=1$ for generic $v \in V$.  

We first consider involutions.   We recall that an involution $g \in \SO_{2n}=\SO(W)$ (in characteristic $2$) is essentially determined
by the number $r$ of nontrivial Jordan blocks of $g$ (equivalently $r = \dim (g-1)W$) and whether the subspace $(g-1)W$ is totally singular or not with $r$ even
(and $r \le n$) --- see \cite{AschbacherSeitz}, \cite{LiebeckSeitz} or
see \cite[Sections 5,6]{FGS} for a quick elementary treatment.   If $r < n$ or $(g-1)V$ is not totally singular, there is  $1$ class
for each possible pair of invariants.  If $r=n$ (and so $n$ is even) and $(g-1)V$ is totally singular, then there are two such classes
interchanged by a graph automorphism of order $2$.   

\begin{lem} Suppose $\car k = 2$.  Let $G=\Spin_{2n}, n >  4$ and let $W$ be a half-spin representation.  If $g \in G$ is an involution  
other than a long root element, then  $\dim W^g  \le (5/8) \dim W$.
\end{lem}

\begin{proof}  By passing
to closures, we may assume that $r=4$.   Thus, $g \in \Spin_8 \le G$.   There is the class of long root elements
(which have $r=2$).  The largest class is invariant under graph automorphisms
and so has a $4$-dimensional fixed space on each of the three $8$ dimensional representations.   The other three
classes are permuted by the graph automorphisms.  Thus, it follows they have a $4$ dimensional fixed space
on two of the $8$-dimensional representations and a $5$-dimensional fixed space on the third such representation.
  Since the class of $g$ is invariant under triality, $g$
has a $4$-dimensional fixed space on each of the $8$-dimensional representations of $\Spin_8$.   Since the
a half-spin representation of $\Spin_{10}$ is a sum of two distinct half-spin representations for $\Spin_8$,
the result is true for $n=5$. 

  The result now follows by
induction (since $W$ is a direct sum of the two half spin representations of $\Spin_{2n-2}$).   
\end{proof}

\begin{lemma}  Suppose $\car k = 2$.  Let $G=\Spin_{18}$ or $\HSpin_{20}$ with $V$ a half spin representation of dimension $256$ or $512$ respectively.  
Then $G_v(k)=1$ for generic $v \in V$.
\end{lemma}

\begin{proof}  This is proved in \cite{GurLawther} but we give a different proof.  As in the previous section, it suffices to show that
$\dim V^g  + \dim g^G < \dim V$ for every non-central $g \in G$ with $g$ of prime order.  If $g$ has odd prime order (and so is semisimple),
then the argument is exactly the same as in the previous section (indeed, it is even easier since there are no involutions
to consider).    Alternatively, since we know the result in characteristic $0$, it
follows that generic stabilizers have no nontrivial semisimple elements as in the proof of \cite[Lemma 10.3]{GG:simple}.

Thus, it suffices to consider $g$ of order $2$.   Let $r$ be the number of nontrivial Jordan blocks of $g$.  If 
$g$ is not a long root element, then 
  $\dim V^g  \le (5/8) \dim V$.  On the other hand,  $\dim g^G \le 99$ for $n=10$ and $79$ for $n=9$ by \cite{AschbacherSeitz}, \cite{LiebeckSeitz}, or  \cite{FGS}; in either case $\dim g^G < (3/8)\dim V$.
  
  The remaining case to consider is when $g$ is a long root element.   Then $\dim V^g = (3/4) \dim V$ while
  $\dim g^G = 34$ or $30$ respectively and again the inequality holds. 
\end{proof}

%%%%%%%%%%%%%%%%%%%%%%%%%%%%%%%%%%%%%%%%%%%%%%%%%%%%%%%%%%%%%%%%%%%%%%%%
\section{Proof of Corollary \ref{ed}} \label{ed.sec}

For $n$ not divisible by 4, the (half) spin representation $\Spin_n$ is generically free by Theorem \ref{MT}, so by, e.g., \cite[Th.~3.13]{M:ed} we have:
\[
\ed(\Spin_n) \le \dim V - \dim \Spin_n.
\]
This gives the upper bound on $\ed(\Spin_n)$ for $n$ not divisible by 4.  For $n = 16$, we use the same calculation with $V$ the direct sum of the vector representation of $\Spin_{16}$ and a half-spin representation.  For $n \ge 20$ and divisible by 4, Theorem \ref{MT} gives that $\ed(\HSpin_n)$ is at most the value claimed; with this in hand, the argument in \cite[Th.~2.2]{ChM:edspin} (referring now to \cite{Loetscher:fiber} instead of \cite{BRV} for the stacky essential dimension inequality) establishes the upper bound on $\ed(\Spin_n)$ for $n \ge 20$ and divisible by 4.

It is trivially true that $\ed_2(\Spin_n) \le \ed(\Spin_n)$.  Finally, that $\ed_2(\Spin_n)$ is at least the expression on the right side of the display was proved in \cite[Th.~3-3(a)]{BRV} for $n$ not divisible by 4 and in  \cite[Th.~4.9]{M:chile}  for $n$ divisible by 4; the lower bound on $\ed_2(\HSpin_n)$ is from \cite[Remark 3-10]{BRV}.  $\hfill\qed$

%%%%%%%%%%%%%%%%%%%%%%%%%%%%%%%%%%%%%%%%%%%%%%%%%%%%%%%%%%%%%%%%%%%%%%
\section{$\Spin_n$ for $6 \le n \le 12$ and characteristic $2$} \label{small.2}

Suppose now that $6 \le n \le 12$ and $\car k = 2$.
Let us now calculate the stabilizer in $\Spin_n$ of a generic vector $v$ in a (half) spin representation, which will justify those entries in Table \ref{stabilizers}.  For $n  = 6$, the $\Spin_6 \cong \SL_4$ and the representation is the natural representation.  For $n = 8$, the half-spin representation is indistinguishable from the vector representation $\Spin_8 \to \SO_8$ and again the claim is clear.  

For the remaining $n$, we verify that the $k$-points $(\Spin_n)_v(k)$ of the generic stabilizer are as claimed, i.e., that the claimed group scheme is the reduced sub-group-scheme of $(\Spin_n)_v$.  The cases $n = 9, 11, 12$ are treated in \cite[Lemma 2.11]{GLMS} and the case $n = 10$ is \cite[p.~496]{Liebeck:affine}.

For $n = 7$, view $\Spin_7$ as the stabilizer of an anisotropic vector in the vector representation of $\Spin_8$; it contains a copy of $G_2$.  As a $G_2$-module, the half-spin representation of $\Spin_8$ is self-dual and has composition factors of dimensions 1, 6, 1, so $G_2$ fixes a vector in $V$.  As $G_2$ is a maximal closed connected subgroup of $\Spin_7$, it is the identity component of the reduced subgroup of $(\Spin_7)_v$.

We have verified that the reduced sub-group-scheme of $(\Spin_n)_v$ agrees with the corresponding entry, call it $S$, in Table \ref{stabilizers}.  We now proceed as in \S\ref{spin20.2} and find a $w$ such that $\dim (\spin_n)_w = \dim S$, which shows that $(\Spin_n)_v$ is smooth, completing the proof of Table \ref{stabilizers} for $n \le 12$.

%%%%%%%%%%%%%%%%%%%%%%%%%%%%%%%%%%%%%%%%%%%%%%%%%%%%%%%%%%%%%%%%%%%%%%%%%%
\section{$\Spin_{13}$ and $\Spin_{14}$ and characteristic $\ne 2$}  \label{stab1314}

In this section, we determine the stabilizer in $\Spin_{14}$ and $\Spin_{13}$ of a generic vector in the (half) spin representation $V$ of dimension 64.
We assume that $\car k \ne 2$ and $k$ is algebraically closed.

Let $C_0$ denote the trace zero subspace of an octonion algebra with quadratic norm $N$.  We may view the natural representation of $\SO_{14}$ as a sum $C_0 \oplus C_0$ endowed with the quadratic form $N \oplus -N$.  This gives an inclusion $G_2 \times G_2 \subset \SO_{14}$ that lifts to an inclusion $G_2 \times G_2 \subset \Spin_{14}$.  There is an element of order 4 in $\SO_{14}$ such that conjugation by it interchanges the two copies of $G_2$ --- the element of order 2 in the orthogonal group with this property has determinant $-1$ --- so the normalizer of $G_2 \times G_2$ in $\SO_{14}(k)$ is isomorphic to $((G_2 \times G_2) \rtimes \mu_4)(k)$ and in $\Spin_{14}$ it is $((G_2 \times G_2) \rtimes \mu_8)(k)$.

Viewing $V$ as an internal Chevalley module for $\Spin_{14}$ (arising from the embedding of $\Spin_{14}$ in $E_8$), it follows that $\Spin_{14}$ has an open orbit in $\mathbb{P}(V)$.  Moreover, the unique $(G_2 \times G_2)$-fixed line $kv$ in $V$ belongs to this open orbit, see \cite[p.~225, Prop.~11]{Popov:14}, \cite{Rost:14.1}, or \cite[\S21]{G:lens}.  That is, for $H$ the reduced sub-group-scheme of $(\Spin_{14})_v$, $H^\circ \supseteq G_2 \times G_2$.  By dimension count this is an equality.  A computation analogous to the one in the preceding paragraph shows that the idealizer of $\Lie(G_2 \times G_2)$ in $\so_{14}$ is $\Lie(G_2 \times G_2)$ itself, hence $\Lie((\Spin_{14})_v) = \Lie(H^\circ)$, i.e., $(\Spin_{14})_v$ is smooth.  It follows from the construction above that the stabilizer of $kv$ in $\Spin_{14}$ is all of $(G_2 \times G_2) \rtimes \mu_8$ (as a group scheme).  The element of order 2 in $\mu_8$ is in the center of $\Spin_{14}$ and acts as $-1$ on $V$, so the stabilizer of $v$ is $G_2 \times G_2$ as claimed in Table \ref{stabilizers}.

Now fix a vector $(c, c') \in C_0 \oplus C_0$ so that $N(c)$, $N(c')$ and $N(c) - N(c')$ are all nonzero.  The stabilizer of $(c,c')$ in $\Spin_{14}$ is a copy of $\Spin_{13}$, and the stabilizer of $v$ in $\Spin_{13}$ is its intersection with $G_2 \times G_2$, i.e., the product $(G_2)_c \times (G_2)_{c'}$.  Each term in the product is a copy of $\SL_3$ (see for example \cite[p.~507, Exercise 6]{KMRT}), as claimed in Table \ref{stabilizers}.  (On the level of Lie algebras and under the additional hypothesis that $\car k = 0$, this was shown by Kac and Vinberg in \cite[\S3.2]{GV13}.)
%%%%%%%%%%%%%%%%%%%%%%%%%%%%%%%%%%%%%%%%%%%%%%%%%%%%%%%%%%%%%%%%%%%%%%%%
\section{$\Spin_{13}$ and $\Spin_{14}$ and characteristic 2} \label{D7.2.sec}

We will calculate the stabilizer in $\Spin_n$ of a generic vector in an irreducible (half-)spin representation for $n = 13, 14$ over a field $k$ of characteristic 2.

\begin{prop} \label{D7.2}
The stabilizer in $\Spin_{14}$ (over a field $k$ of characteristic 2) of a generic vector in a half-spin representation is the group scheme $(G_2 \times G_2)\rtimes \Z/2$.
\end{prop}

We use the following construction.  Let $X \supset R$, $V_1$, $V_2$ be vector spaces endowed with quadratic forms $q_X$, $q_R := q_X\vert_R$, $q_1$, $q_1$ such that $q_R$ is totally singular; $q_X$, $q_1$, and $q_2$ are nonsingular; $R$ is a maximal totally singular subspace of $X$; and there exist isometric embeddings $f_i \!: (X, q_X) \injects (V_i, q_i)$.   For example, one could take $V_1$ and $V_2$ to be copies of an octonion algebra $C$, $R$ to be the span of the identity element $1_C$, and $X$ to be a quadratic \'etale subalgebra of $C$.  There is a natural quadratic form on the pushout $(V_1 \oplus V_2)/(f_1 -f_2)(X)$; if we write $V_i \cong V'_i \perp f_i(X)$, then the quadratic space is isomorphic to $V'_1 \perp V'_2 \perp X$.  We can perform a similar construction where the role of $V_i$ is played by the codimension-1 subspace $f_i(R)^\perp$ and the pushout is $(f_1(R)^\perp \oplus f_2(R)^\perp)/(f_1 - f_2)(R)$, giving a homomorphism of algebraic groups $B_{\ell_1} \times B_{\ell_2} \to B_{\ell_1 + \ell_2}$ where $2\ell_i + 2 = \dim V_i$.

\begin{proof}[Proof of Proposition \ref{D7.2}]
The 7-dimensional Weyl module of the split $G_2$ gives an embedding $G_2 \injects \SO_7$.  Combining this with the construction in the previous paragraph gives maps
\[
G_2 \times G_2 \to \SO_7 \times \SO_7 \to \SO_{13} \to \SO_{14}
\]
which lift to maps where every $\SO$ is replaced by $\Spin$.

Put $V$ for a half-spin representation of $\Spin_{14}$.  It restricts to the spin representation of $\Spin_{13}$.  Calculating the restriction of the weights of $V$ to $\Spin_7 \times \Spin_7$ using the explicit description of the embedding, we see that $V$ is the tensor product of the 8-dimensional spin representations of $\Spin_7$.  By triality, the restriction of one of the spin representations to $G_2$ is the action of $G_2$ on the octonions $C$, which is a uniserial module with 1-dimensional socle $S$ (spanned by the identity element in $C$) and 7-dimensional radical, the Weyl module of trace zero octonions.  The restriction of $V = C \ot C$ to the first copy of $G_2$ is eight copies of $C$, so has an 8-dimensional fixed space $S \ot C$.  As $(S \ot C)^{1 \times G_2} = S \ot S$, we find that $S \ot S$ is the unique line in $V$ stabilized by $G_2 \times G_2$.

We now argue that the $\Spin_{14}$-orbit of $S \ot S$ is open in $\mathbb{P}(V)$.  To see this, by \cite{Roe:certain}, it suffices to verify that $G_2 \times G_2$ is not contained in the Levi subgroup of a parabolic subgroup of $\Spin_{14}$.  This is easily verified; the most interesting case is where the Levi has type $A_6$, and $G_2 \times G_2$ cannot be contained in such because the restriction of $V$ to $A_6$ has composition factors of dimension 1, 7, 21, and 35.  We conclude that every nonzero $v \in S \ot S$ is a generic vector in $V$ and $(\Spin_{14})_v$ has dimension 28.

If one constructs on a computer the representation $V$ of the Lie algebra $\spin_{14}$ over a finite field $F$ of characteristic 2, then it is a matter of linear algebra to calculate the dimension of the stabilizer $(\spin_{14})_x$ of a random vector $x \in V$.  One finds for some $x$ that the stabilizer has dimension 28, which is the minimum possible, so by semicontinuity of dimension $\dim((\spin_{14})_v) = 28 = \dim(G_2 \times G_2)$.  That is, $(\Spin_{14})_v$ is smooth with identity component $G_2 \times G_2$.  Consequently we may compute $(\Spin_{14})_v$ by determining its $K$-points for $K$ an algebraic closure of $k$.  The map $\Spin_{14}(K) \to \SO_{14}(K)$ is an isomorphism of concrete groups.  The normalizer of $(G_2 \times G_2)(K)$ in the latter group is $(G_2 \times G_2)(K) \rtimes \Z/2$, where the nonidentity element $\tau \in \Z/2$ interchanges the two copies of $\SO_7(K)$, hence of $G_2(K)$.  As $\tau$ normalizes $(G_2 \times G_2)(K)$, it leaves the fixed subspace $S \ot S \ot K = Kv$ invariant, and we find a homomorphism $\chi \!: \Z/2 \to \Gm$ given by $\tau v = \chi(\tau)v$ which must be trivial because $\car K = 2$.
\end{proof}

The above proof, which is somewhat longer than some alternatives, was chosen because of the details it provides on the embedding of $G_2 \times G_2$ in $\Spin_{14}$.

\begin{prop} \label{B6.2}
The stabilizer in $\Spin_{13}$ (over a field of characteristic 2) of a generic vector in the spin representation 
is the group scheme $(\SL_2 \times \SL_2) \rtimes \Z/2$.
\end{prop}

\begin{proof}
We imitate the argument used in \S\ref{stab1314}.  View $\Spin_{13}$ as $(\Spin_{14})_y$ for an anisotropic $y$ in the 14-dimensional vector representation of $\Spin_{14}$.  That representation, as a representation of $\Spin_{13}$, has socle $ky$ and radical $y^\perp$.  Let $v$ be a generic element of the spin representation $V$ of $\Spin_{13}$.  Our task is to determine the group
\begin{equation} \label{13.int}
(\Spin_{13})_v = (\Spin_{14})_y \cap (\Spin_{14})_v.
\end{equation}

The stabilizer $(\Spin_{14})_v$ described above is contained in a copy $(\Spin_{14})_e$ of $\Spin_{13}$ where $ke$ is the radical of the 13-dimensional quadratic form given by the pushout construction.  As $v$ is generic, $y$ and $e$ are in general position, so tracing through the pushout construction we see that the intersection \eqref{13.int} contains the product of 2 copies of the stabilizer in $G_2$ of a generic octonion $z$.  The quadratic \'etale subalgebra of $C$ generated by $z$ has normalizer $\SL_3 \rtimes \Z/2$ in $G_2$, hence the stabilizer of $z$ is $\SL_3$.  We conclude that, for $K$ an algebraic closure of $k$,  the group of $K$-points of $(\Spin_{13})_v$ equals that of the claimed group, hence the stabilizer has dimension 16.  Calculating with a computer as in the proof for $\Spin_{14}$, we find that $\dim (\spin_{13})_v \le 16$, and therefore the stabilizer of $v$ is smooth as claimed.
\end{proof}

%%%%%%%%%%%%%%%%%%%%%%%%%%%%%%%%%%%%%%%%%%%%%%%%
\appendix
\section{Generic stabilizers associated with a peculiar half-spin representation}
\begin{center}
\emph{by Alexander Premet}
\end{center}

\setcounter{subsection}{-1}
\subsection{The main theorem} Throughout this appendix we work over an algebraically closed field $k$ of characteristic $2$.
Let $G=\HSpin_{16}(k)$ and let  $V$ be the natural (half-spin) $G$-module. The theorem stated below describes the generic stabilizers for the actions of $G$ and $\g=\Lie(G)$ on $V$.

\begin{thm} \label{Premet}
The following are true:
\begin{enumerate}
\renewcommand{\theenumi}{\roman{enumi}}
\item \label{Premet.1}
There exists a non-empty Zariski open subset $U$ in $V$ such that for every $x\in U$ the stabilizer $G_x$ is isomorphic to $(\Z/2\Z)^4$.

\item \label{Premet.2}
For any $x\in U$ the stabilizer $\g_x$ is a $4$-dimensional toral subalgebra of $\g$.

\item \label{Premet.3}
If $x,x'\in U$ then the stabilizers $G_x$ and $G_{x'}$ and the infinitesimal stabilizers $\g_x$ and $\g_{x'}$ are $G$-conjugate.

\item \label{Premet.4}
The scheme-theoretic stabilizer of any $x\in U$ is isomorphic to $(\Z/2\Z)^4\times (\mu_2)^4$.
\end{enumerate}
\end{thm}

A more precise description of $G_x$ and $\g_x$ with $x\in U$ is given in \S\ref{Premet.A4}. It should be mentioned here that our Theorem~\ref{Premet} can also be deduced from more general invariant-theoretic results recently announced by Eric Rains.

\subsection{Preliminary remarks and recollections} \label{Premet.A1}
Let $\Gt$ be a simple algebraic group of type $\rm{E}_8$ over $k$ and $\gt={\rm Lie}(\Gt)$. The Lie algebra $\gt$ is simple and carries an
$(\Ad\,G)$-equivariant $[p]$-th power map $x\mapsto x^{[p]}$. Since $p=2$, Jacobson's formula for $[p]$-th powers is surprisingly simple: we have that
\[
(x+y)^{[2]}=x^{[2]}+y^{[2]}+[x,y]\qquad \left(\forall\, x,y\in \gt\right).
\]
Let $T$ be a maximal torus of $\Gt$ and $\t=\Lie(T)$. Write $\tilde{\Phi}$ for the root system of $\Gt$ with respect to $T$.  In what follows we will make essential use of Bourbaki's description of roots in $\tilde{\Phi}$; see \cite[Planche~VII]{Bou:g4}. More precisely, let
$\mathbf{E}$ be an $8$-dimensional Euclidean space over $\mathbb{R}$
with orthonormal basis $\{\varepsilon_1,\ldots,\varepsilon_8\}$. Then $\tilde{\Phi}=\Tilde{\Phi}_0\sqcup \tilde{\Phi}_{1}$ where 
\[
\tilde{\Phi}_0=\left\{\pm\varepsilon_i\pm\varepsilon_j \mid 1\le i<j\le 8\right\}
\]
and 
\[
\tilde{\Phi}_{1}=\left\{\textstyle{\frac{1}{2}}\textstyle{\sum}_{i=1}^8\,(-1)^{\nu(i)}\varepsilon_i \mid \textstyle{\sum}_{i=1}^8\,\nu(i)\in 2\Z\right\}.
\]
The roots $\alpha_1=\frac{1}{2}(\varepsilon_1+\varepsilon_8-\varepsilon_2-\varepsilon_3-
\varepsilon_4
-\varepsilon_5-\varepsilon_6-\varepsilon_7)$,
$\alpha_2=\varepsilon_1+\varepsilon_2$,
$\alpha_3=\varepsilon_2-\varepsilon_1$,
$\alpha_4=\varepsilon_3-\varepsilon_2$,
$\alpha_5=\varepsilon_4-\varepsilon_3$,
$\alpha_6=\varepsilon_5-\varepsilon_4$,
$\alpha_7=\varepsilon_6-\varepsilon_5$,
$\alpha_8=\varepsilon_7-\varepsilon_6$
form a basis of simple roots in $\tilde{\Phi}$ which we denote by $\tilde{\Pi}$.
\noindent
Let $(\,\cdot\,\vert\,\cdot\,)$ be the scalar product of $\mathbf{E}$. It is invariant under the action of the Weyl
group $W(\tilde{\Phi})\subset {\rm GL}(\mathbf{E})$.

Given $\alpha\in\tilde{\Phi}$ we denote by $U_\alpha$ and $e_\alpha$ the unipotent root subgroup of $\Gt$ and a root vector in $\Lie(U_\alpha)$. Let $V$ be the $k$-span of
of all $e_\alpha$ with $\alpha\in\tilde{\Phi}_{1}$ and write $G$ for the subgroup of $\Gt$ generated by $T$ and all $U_\alpha$ with $\alpha\in\tilde{\Phi}_0$. It is well known (and straightforward to see) that the algebraic $k$-group $G$ is isomorphic to $\HSpin_{16}(k)$ and the $G$-stable subspace $V$ of $\gt$ is isomorphic to the natural  (half-spin) $G$-module: one can choose a Borel subgroup $B$ of $G$ in such a way that the fixed-point space $V^{R_u(B)}$ is spanned by $e_{-\alpha_1}$. We write $W$ for the subgroup of $W(\tilde{\Phi})$ generated all orthogonal reflections $s_\alpha$ with $\alpha\in\tilde{\Phi}_0$. Clearly, $W\cong N_G(T)/T$ is the Weyl group of $G$ relative to $T$. Since $G$ has type ${\rm D}_8$ the group $W$ is a semidirect product of its subgroup $W_0\cong\mathfrak{S}_8$ acting by permutations of the set $\{\varepsilon_1,\ldots,\varepsilon_8\}$
and its abelian normal subgroup $A\cong (\Z/2\Z)^7$
consisting of all maps $\varepsilon_i\mapsto (\pm 1)_i\varepsilon_i$ with $\prod_{i=1}^8(\pm 1)_i=1$; see \cite[Planche~IV]{Bou:g4}.

We may (and will) assume further that the $e_\alpha$'s are obtained by base change from a Chevalley $\Z$-form, $\gt_\Z$, of a complex Lie algebra of type ${\rm E}_8$. Since the group $\Gt$ is a simply connected the nonzero elements  $h_\alpha:=[e_\alpha,e_{-\alpha}]\in\t$  with $\alpha\in\tilde{\Phi}$ span $\t$. They have the property that $[h_\alpha,e_{\pm \alpha}]=\pm 2e_{\pm\alpha}=0$ and $h_\alpha=h_{-\alpha}$ for all $\alpha\in\tilde{\Phi}$.
It is well known that $e_\alpha^{[2]}=0$ and $h_\alpha^{[2]}=h_\alpha$ for all
$\alpha\in\tilde{\Phi}$.
The set $\{h_\alpha \mid \alpha\in\tilde{\Pi}\}$ is a $k$-basis of $\t$. Since $\gt$ is a simple Lie algebra, for every nonzero $t\in\t$ there is a simple root
$\beta\in\tilde{\Pi}$ such that $({\rm d}\beta)_e(t)\ne 0$. This implies that $\t$ admits a non-degenerate $W(\tilde{\Phi})$-invariant symplectic bilinear form
$\langle\,\cdot\,,\,\cdot\,\rangle$ such that
$\langle h_\alpha,h_\beta\rangle=(\alpha \vert \beta)\!
\mod 2$ for all
$\alpha,\beta\in\tilde{\Phi}.$

\subsection{Orthogonal half-spin roots and
Hadamard--Sylvester matrices} \label{Premet.A2}
Following the Wikipedia webpage on Hadamard matrices
we define the matrices $H_{2^k}$ of order $2^k$, where $k\in\Z_{\ge 0}$,  by setting $H_1=[1]$ and
$$H_{2^{k+1}}=\left[\begin{array}{rr}H_{2^k}&H_{2^k}\\
H_{2^k}&-H_{2^k}\end{array}\right]=H_{2}\otimes H_{2^k}$$ for $k\ge 0$.
These Hadamard matrices were first introduced by Sylvester in 1867 and they have the property that $H_{2^k}\cdot H_{2^k}^{\rm T}=2^k\cdot I_{2^k}$ for all $k$.
We are mostly interested in
$$
H_8=H_2\otimes H_2\otimes H_2=\left[\begin{array}{rrrrrrrr}1&1&1&1&1&1&1&1\\
1&-1&1&-1&1&-1&1&-1\\
1&1&-1&-1&1&1&-1&-1\\
1&-1&-1&1&1&-1&-1&1\\
1&1&1&1&-1&-1&-1&-1\\
1&-1&1&-1&-1&1&-1&1\\
1&1&-1&-1&-1&-1&1&1\\
1&-1&-1&1&-1&1&1&-1
\end{array}\right].$$
To each row $r_i=(r_{i1},\ldots, r_{i8})$ of $H_8$ we assign the root $\gamma_i=\frac{1}{2}(r_{i1}\varepsilon_1+\cdots+r_{i8}\varepsilon_8)$. This way we obtain $16$ distinct roots  $\pm \gamma_1,\ldots,\pm\gamma_8$ in $\tilde{\Phi}_1$ with the property that $(\gamma_i\vert\gamma_j)=0$ for all $i\ne j$.
As $\pm\gamma_i\pm \gamma_j\not\in\tilde{\Phi}$ for $i\ne j$, the semisimple regular subgroup $S$ of $\Gt$ generated by $T$ and all $U_{\pm \gamma_i}$ is connected and has type ${\rm A}_1^8$.
It is immediate from the Bruhat decomposition in $\Gt$ that
$G\cap S=N_G(T)\cap N_S(T).$

Using the explicit form of the simple roots $\alpha_1,\ldots,\alpha_8$ it is routine to determine  the matrix
$M:=\big[(\gamma_i\vert\alpha_j)\big]_{1\le i,j\le 8}$. It has the following form:
$$
M=\left[\begin{array}{rrrrrrrr}
-1&1&0&0&0&0&0&0\\
0&0&-1&1&-1&1&-1&1\\
0&1&0&-1&0&1&0&-1\\
1&0&-1&0&1&0&-1&0\\
0&1&0&0&0&-1&0&0\\
1&0&-1&1&-1&0&1&-1\\
1&1&0&-1&0&0&0&1\\
0&0&-1&0&1&-1&1&0
\end{array}\right].$$ It is then straightforward to check that $M$ is row-equivalent over the integers to
a block-triangular matrix
$M'=\left[\begin{array}{cc}
M_1& M_2\\
O_4&2M_3
\end{array}\right]$ with $M_1,M_2,M_3\in{\rm Mat}_4(\Z)$ and
$\det(M_1)=\det(M_3)=1$.
From this it follows that $\gamma_1,\ldots,\gamma_8$ span $\mathbf{E}$ over $\mathbb{R}$ and  $h_{\gamma_1},\ldots,h_{\gamma_8}$ span a maximal
($4$-dimensional) totally isotropic subspace of the symplectic space $\t$. We call it $\t_0$.

\subsection{A dominant morphism} \label{Premet.A3}
 Put $\Gamma=\{\gamma_1,\ldots,\gamma_8\}$ and let $\mathfrak{r}$ denote the subspace of $V$ spanned by $e_\gamma$ with $\gamma\in\pm\Gamma$. If $x=\sum_{i=1}^8(\lambda_ie_{\gamma_i}+
\mu_i e_{-\gamma_i})\in \mathfrak{r}$ then
Jacobson's formula shows that $x^{[2]}=\sum_{i=1}^8(\lambda_i\mu_i)h_{\gamma_i}\in\t_0$.
It follows that
\begin{equation}\label{p-th} x^{[2]^{k+1}}=\,\textstyle{\sum}_{i=1}^8
(\lambda_i\mu_i)^{2^k}
h_{\gamma_i}\qquad\big(\forall\, k\ge 0\big).
\end{equation}
Our discussion at the end of \S\ref{Premet.A2} shows that $\t_0$ has a basis $t_1,\ldots, t_4$ contained in the $\mathbb{F}_2$-span of $\{h_\gamma\,|\,\,\gamma\in\Gamma\}$. Since $h_\alpha^{[2]}=h_\alpha$ for all roots $\alpha$,  we have that $t_i^{[2]}=t_i$ for $1\le i\le 4$.
In view of (\ref{p-th}) this yields that the subset of $\mathfrak{r}$ consisting of all
$x$ as above such that
$\lambda_i\mu_i\ne \lambda_j\mu_j$ for $i\ne j$ and $\{x^{[2]^k}\,|\,\,1\le k\le
4\}$ spans $\t_0$ is non-empty and Zariski open in $\mathfrak{r}$. We call this subset $\mathfrak{r}^\circ$ and consider the morphism
$$\psi\colon\,G\times \mathfrak{r}\longrightarrow\, V,\ \ \,(g,x)\mapsto\,(\Ad\, g)\cdot x.$$ Note that
$\dim (G\times\mathfrak{r})=120+16=136$ and $\dim V=128$.
By the theorem on fiber dimensions of a morphism, in order to show that $\psi$ is dominant it suffices to find a point $(g,x)\in G\times \mathfrak{r}$ such that all components of $\psi^{-1}((\Ad\,g)\cdot x)$ containing $(g,x)$ have dimension $\le 8$.

We take $x\in\mathfrak{r}^\circ$ and $g=1_{\Gt}$. Clearly, $\psi^{-1}(x)\subset
\{(g,y)\in G\times\mathfrak{r}\,|\,\,y\in (\Ad\, G)\cdot x\}$.
If $(g,y)\in\psi^{-1}(x)$ then
$y\in\mathfrak{r}$ and $(\Ad\, g)^{-1}$ maps the $k$-span, $\t(x)$, of $\{x^{[2]^k}\,|\,\,1\le k\le 4\}$ onto 
the $k$-span, $\t(y)$, of $\{y^{[2]^k}\,|\,\,1\le k\le 4\}$. As $y^{[2]}\in\t_0$ and $\t_0$ is a restricted subalgebra of $\t$, this implies that
$\t(x)=\t(y)=\t_0$.
It follows that $\Ad\,g$ preserves  the Lie subalgebra $\cc_{\g}(\t_0)$ of $\g$. The centralizer $\cc_{\gt}(\t_0)$ is spanned by $\t$ and all root vectors $e_\alpha$ such that
$\langle h_\alpha,h_{\gamma_i}\rangle=({\rm d}\alpha)_e(h_{\gamma_i})=0$ for $1\le i\le 8$.
As $\t_0$ is a maximal totally isotropic subspace of the symplectic space $\t$, our concluding remark in \S\ref{Premet.A2} shows that $\cc_{\gt}(\t_0)=\Lie(S)$. Since $\cc_\g(\t_0)=\g\cap\Lie(S)=\t$ we obtain that $g\in N_G(T)$. But then $\psi^{-1}(x)\subseteq \{(g,(\Ad\,g)^{-1}\cdot x)\in G\times\mathfrak{r}^\circ\,|\,\,g\in N_G(T)\}$. Since $\dim N_G(T)=\dim T=8$, {\it all} irreducible components of $\psi^{-1}(x)$ have dimension $\le 8$. We thus deduce that the morphism $\psi$ is dominant. As the set $G\times \mathfrak{r}^\circ$ is Zariski open in $G\times \mathfrak{r}$, the
$G$-saturation of $\mathfrak{r}^\circ$ in $V$ contains a Zariski open subset of $V$.

 \subsection{Generic stabilizers} \label{Premet.A4}
Let $x=\sum_{i=1}^8\lambda_ie_{\gamma_i}+\sum_{i=1}^8
\mu_ie_{-\gamma_i}\in\mathfrak{r}^\circ$.
In view of our discussion in \S\ref{Premet.A3} we now need to determine the
stabilizer $G_x$. If $g\in G_x$ then $\Ad\,g$ fixes $\t_0={\rm span}\{x^{[2]^i}\,|\,\,1\le i\le 4\}$ and hence preserves $\cc_\g(\t_0)=\t$. This yields
$G_x\subseteq N_G(T)$. Working over a field of characteristic $2$ has some advantages: after reduction modulo $2$ we are no longer affected by the ambiguity in the choice of a Chevalley basis in $\gt_\Z$ and the torus $T$ has no elements of order $2$. It follows that
$N_{\Gt}(T)$ contains a subgroup
isomorphic $W(\tilde{\Phi})$ which intersects trivially with $T$.
In the notation of \cite[\S 3]{St} this group is generated by all elements $\omega_\alpha=w_\alpha(1)$ with $\alpha\in\tilde{\Phi}$. As a consequence, $W$ embeds into $N_G(T)$ in such a way that
$N_G(T)=W\ltimes T$.

Our discussion in \S\ref{Premet.A2} implies that for any
$\alpha\in\tilde{\Pi}$ the element $16\alpha\in\Z\tilde{\Phi}$ lies in
the $\Z$-span of $\gamma_1\ldots,\gamma_8$. Since $T$ has no elements of order $2$ and $\Gt$ is a group of adjoint type, it follows that
for any collection $(t_1,\ldots,t_8)\in (k^\times)^8$ there exists a unique element $h=h(t_1,\ldots,t_8)\in T$ with $\gamma_i(h)=t_i$ for all $1\le i\le 8$. Conversely, any element of $T$ has this form. As a consequence, $\Gt_x\cap T=\{1_{\Gt}\}$.
For $1\le i\le 8$ we set $h_i:=h(1,\ldots,\mu_i/\lambda_i,\ldots,1)$, an element of $T$, where the entry $\mu_i/\lambda_i$ occupies the $i$-th position.
Since $\Ad\,s_{\gamma_i}$ permutes $e_{\pm\gamma_i}$ and fixes $e_{\pm\gamma_j}$ with $j\ne i$, it is straightforward to check that $s_{\gamma_i} h_i\in\Gt_x$.
If $w_0$ is the longest element of $W(\tilde{\Phi})$
then it acts on $\Z\tilde{\Phi}$ as $-{\rm Id}$ and hence lies in $A\subset W\hookrightarrow N_G(T)$. Since $w_0=\prod_{i=1}^8s_{\gamma_i}$ we now deduce that
$n_0:=w_0\big(\prod_{i=1}^8 h_i\big)\in G_x$.

Suppose $\Gt_x\cap N_{\Gt}(T)$ contains an element $n=wh$, where $w\in W(\tilde{\Phi})$ and $h=h(a_1,\ldots, a_8)\in T$, such that $w(\gamma_i)=\gamma_j$ for $i\ne j$. Then $n(e_{\gamma_i})=a_ie_{\gamma_j}$ and
$n(e_{-\gamma_i})=a_i^{-1}e_{-\gamma_j}$ implying that
$\lambda_j=\lambda_ia_i$ and $\mu_j=\mu_ia_i^{-1}$. But then $\lambda_j/\lambda_i=\mu_i/\mu_j$ forcing $\lambda_i\mu_i=\lambda_j\mu_j$ for $i\ne j$. Since $x\in\mathfrak{r}^\circ$ this is false. As $n_0\in G_x$ and $w_0(\pm\gamma_i)=\mp\gamma_i$ for all $i$, this argument shows that $\Gt_x\cap N_{\Gt}(T)=\langle
n_i\,|\,\,1\le i\le 8\rangle$ is isomorphic to an elementary abelian $2$-group of order $2^8$.

Let $\mathcal{A}_{2^k}\cong (\Z/2\Z)^{2^k}$ denote the direct product of $2^k$ copies of $\{\pm 1\}\cong \Z/2\Z$. The group operation
in $\mathcal{A}_{2^k}$ is defined
componentwise. We write $u\bullet v$ for the product of $u,v\in\mathcal{A}_{2^k}$ and denote by $\mathbf{1}_{2^n}$ the identity element of $\mathcal{A}_{2^k}$ (all components of $\mathbf{1}_{2^k}$ are equal to $1$). The set of rows, $R_{2^k}$, of the Hadamard--Sylvester matrix  $H_{2^k}$ may be regarded as a subset $\mathcal{A}_{2^k}$ and easy induction on $k$ shows that
$\pm R_{2^k}$ is a subgroup of $\mathcal{A}_{2^k}$. In particular, $\pm R_8$ is a subgroup of  $\mathcal{A}_8$. As mentioned in \S\ref{Premet.A1} the subgroup $W_0\cong\mathfrak{S}_8$ of the Weyl group $W=W(\tilde{\Phi}_0)$ acts on $\mathcal{A}_8$ by permuting components whereas the normal subgroup
$A\cong(\Z/2\Z)^7$ of $W$ embeds into $\mathcal{A}_8$ and acts on it by translations.

If $n\in G_x$ then $n=wh\in N_G(T)$ and
$w$ preserves $\pm R_8$ setwise.
If $w=a\sigma$, where $\sigma\in W_0$ and $a\in A$, then our discussion in the previous paragraph shows that 
$w(u)=(a\sigma)(u)=\pm u$ for all $u\in \pm R_8$.
Taking $u=\mathbf{1}_8$ we get $\sigma(\mathbf{1}_8)=\mathbf{1}_8$ and $\pm \mathbf{1}_8=w(\mathbf{1}_8)=a\bullet \sigma(\mathbf{1}_8)=a\bullet \mathbf{1}_8=a$. This yields $a=\pm \mathbf{1}_8$
implying that $w\in W_0$ preserves $\pm R_8$. Also, $G_x\cap A$ is a cyclic group of order $2$ generated by $n_0$.

We now consider three commuting involutions $\sigma_1=(1,5)(2,6)(3,7)(4,8)$, $\sigma_2=(1,4)(2,3)(5,8)(6,7)$ and $\sigma_3=(1,2)(3,4)(5,6)(7,8)$ in $W\cong \mathfrak{S}_8$.  One can see by inspection that each of them maps every $r\in R_8$ to $\pm r$.
Hence $\sigma_i\in\langle s_{\gamma_i}\,|\,\,1\le i\le 8\rangle$.
Since $s_{\gamma_i}h_i\in\Gt_x$ for $1\le i\le 8$,  each $\sigma_i$  admits a unique lift
in $G_x\subset N_G(T)$ which will be denoted by $n_i$.
The subgroup $\langle n_i\,|\, 0\le i\le 3\rangle$ of $G_x$ is isomorphic to $(\Z/2\Z)^4$.

Next we show that any element $\sigma h\in G_x$ with $\sigma\in W_0\cong\mathfrak{S}_8$ lies in the subgroup generated by the $n_i$'s. Since $w$ maps $\mathbf{1}_8$ to $\pm\mathbf{1}_8$ and $n_0\in G_x$ we may assume that $w(\mathbf{1}_8)=\mathbf{1}_8$. Since $\sigma$ maps $(\mathbf{1}_4,-\mathbf{1}_4)$ to $\pm(\mathbf{1}_4,-\mathbf{1}_4)$ and $n_1\in G_x$
we may also assume that $\sigma$ fixes $(\mathbf{1}_4,-\mathbf{1}_4)$. Since $\sigma$ maps $(\mathbf{1}_2,-\mathbf{1}_2,\mathbf{1}_2,-\mathbf{1}_2)$ to $\pm(\mathbf{1}_2,-\mathbf{1}_2,\mathbf{1}_2,-\mathbf{1}_2)$ and $n_2\in G_x$
we may  assume that $\sigma$ fixes $(\mathbf{1}_2,-\mathbf{1}_2,\mathbf{1}_2,-\mathbf{1}_2)$ as well. Finally, since $\sigma$ maps
$(1,-1,1,-1,1,-1,1,-1)$ to $\pm(1,-1,1,-1,1,-1,1,-1)$ and $n_3\in G_x$ we may assume that $\sigma$ fixes $(1,-1,1,-1,1,-1,1,-1)$. This entails that $\sigma(i)=i$ for $i\in\{1,2,3,4\}$. As $\sigma(r)=\pm r$ for all $r\in R_8$ the latter shows that $\sigma={\rm id}$ proving statement \eqref{Premet.1} of Theorem~\ref{Premet}.

Since $\g_x$ contains the spanning set $\{x^{[2]^i}\,|\,\,1\le i\le 4\}$ of $\t_0$, our remarks in \S\ref{Premet.A3} show that $\g_x\subset \t$. Since $[t,x]=0$ for every $t\in\g_x$ it must be that $({\rm d}\gamma)_e(t)=0$ for all $\gamma\in\Gamma$. Since $({\rm d}\gamma)_e(t)=\langle h_\gamma,t\rangle$ and $\t_0$ is a maximal isotropic subspace of the symplectic space $\t$, we obtain that
$t\in\t_0$. As a result, $\g_x=\t_0$ for every $x\in\mathfrak{r}^\circ$. Statement \eqref{Premet.2} follows.

In proving statement \eqref{Premet.3} we may assume that
$x=\sum_{i=1}^8\,(\lambda_i e_{\gamma_i}+\mu_i e_{-\gamma_i})$
and $x'=\sum_{i=1}^8\,(\lambda'_i e_{\gamma_i}+\mu'_i e_{-\gamma_i})$ are two elements of
$\mathfrak{r}^\circ$. Our discussion in the previous paragraph shows that $\g_x=\g_{x'}=\t_0$.
Let $h'_i:=h(1,\ldots,\mu'_i/\lambda'_i,\ldots,1)$, where the entry $\mu'_i/\lambda'_i$ occupies the $i$-th position.
There is a unique element $h=h(b_1,\ldots,b_8)\in T$ such that 
\[
h\cdot s_{\gamma_i}h_i\cdot h^{-1}=s_{\gamma_i}h'_i\qquad \quad (1\le i\le 8).
\]
(We need to take $b_i=\sqrt{(\lambda_i\mu_i')/(\lambda'_i\mu_i)}\in k$
for all $1\le i\le 8$.) Our earlier remarks in this section now show that $h\cdot G_x\cdot h^{-1}=G_{x'}$.
This proves statement \eqref{Premet.3}.

\begin{rmk*} 
We stress that
for an element $x=\sum_{i=1}^8\,(\lambda_i e_{\gamma_i}+\mu_i e_{-\gamma_i})$ to be in $\mathfrak{r}^\circ$ it is necessary that
$\lambda_i\mu_i\ne \lambda_j\mu_j$ for all $i\ne j$.
If one removes this condition and only requires that the set $\{x^{[2]^i} \mid 1 \le i\le 4\}\subset\t$ is linearly independent, then one obtains an a priori bigger Zariski open subset,  $\mathfrak{r}'$, in $\mathfrak{r}$
which still has the property that $G_x$ is a finite group and $\g_x=\t_0$ for every $x\in\mathfrak{r}'$. However, it is not immediately clear that the stabilizers in $G$ of any two elements in $\mathfrak{r}'$ are isomorphic. It would be interesting to investigate this situation in more detail.
\end{rmk*}

\subsection*{Scheme-theoretic stabilizers}
Let $\Gtb$ be a reductive group scheme over $k$ with root system $\tilde{\Phi}$ with respect to a maximal torus $
\mathbf{T}\subset\Gtb$ and let
$\mathbf{G}$ be the regular group subscheme of $\Gtb$ with root system $\tilde{\Phi}_0$.
We may assume that $\mathbf{T}(k)=T$, $\Gtb(k)=\Gt$, and $\mathbf{G}(k)=G$. In this situation, we wish to describe the
scheme-theoretic stabilizer $\mathbf{G}_x$ of $x\in\mathfrak{r}^\circ$, an affine group subscheme of $\mathbf{G}$ defined over $k$.

Let $F$ be any commutative associative $k$-algebra with $1$. The subscheme
$N_{\mathbf{G}}(\mathbf{T})$ of $\mathbf{G}$ is smooth
and since $p=2$ we have an isomorphism $N_{\mathbf{G}}(\mathbf{T})=
W\times \mathbf{T}$ of affine group schemes over $k$.
Arguing as in \S\ref{Premet.A4} one observes
that $\mathbf{G}_x(F)$ is contained
in the group of $F$-points of $N_{\mathbf{G}}(\mathbf{T})$. Since the latter contains $G_x=\mathbf{G}_x(k)$ it follows that
$\mathbf{G}_x(F)$ is generated by $G_x=(\mathbf{G}_x)_{\rm red}$ and the scheme-theoretic stabilizer $\mathbf{T}_x$. More precisely,
$$\mathbf{G}_x\cong (\mathbf{G}_x)_{\rm red}\times \mathbf{T}_x\cong (\Z/2\Z)^4\times \mathbf{T}_x$$
as affine group schemes over $k$.
Our concluding remarks in \S\ref{Premet.A2} imply that
the root lattice $\Z\tilde{\Phi}$ contains free
$\Z$-submodules $\Lambda_1$ and $\Lambda_2$ of rank $4$ such that $\Z\tilde{\Phi}=\Lambda_1\oplus \Lambda_2$
and $\Z\Gamma:=\Z\gamma_1\oplus\cdots \oplus \Z\gamma_8\,=\,\Lambda_1\oplus 2\Lambda_2$. Since $\mathbf{T}(F)=\mathrm{Hom}_{\Z}(\Z\tilde{\Phi},F^\times)$, we have a short exact sequence
$$1\to {\rm Hom}_{\Z}(\Lambda_2/2\Lambda_2,F^\times)\to \mathbf{T}(F)\to {\rm Hom}_{\Z}(\Z\Gamma,F^\times)\to 1$$
which shows that the groups $\mathbf{T}_x(F)$ and ${\rm Hom}_{\Z}(\Lambda_2/2\Lambda_2,F^\times)$ are isomorphic.
Since $\Lambda_2/2\Lambda_2\cong (\Z/2\Z)^4$ and ${\rm Hom}_{\Z}(\Z/2\Z,F^\times)=\mu_2(F)$ we have 
${\rm Hom}_{\Z}(\Lambda_2/2\Lambda_2,F^\times)\cong (\mu_2)^4(F)$. Consequently, $\mathbf{T}_x\cong (\mu_2)^4$ as affine group schemes over $k$. This completes the proof of Theorem~\ref{Premet}. $\hfill\qed$

\bibliographystyle{amsalpha}
\bibliography{skip_master}

\end{document}